\documentclass[11pt,a4paper]{article}

\usepackage[margin=2.5cm]{geometry}
\usepackage{amsmath,amssymb,amsthm,mathtools}
\usepackage[T1]{fontenc}
\usepackage{lmodern}
\usepackage{microtype}
\usepackage{enumitem}
\usepackage[bookmarks=false,colorlinks=true,linkcolor=blue,citecolor=blue,urlcolor=blue]{hyperref}

\theoremstyle{plain}
\newtheorem{theorem}{Theorem}[section]
\newtheorem{proposition}[theorem]{Proposition}
\newtheorem{lemma}[theorem]{Lemma}
\newtheorem{corollary}[theorem]{Corollary}
\theoremstyle{definition}
\newtheorem{definition}[theorem]{Definition}
\newtheorem{example}[theorem]{Example}
\theoremstyle{remark}
\newtheorem{remark}[theorem]{Remark}

\newcommand{\R}{\mathbb{R}}
\newcommand{\Ric}{\mathrm{Ric}}
\newcommand{\Rie}{\mathrm{Rie}}
\newcommand{\Sec}{\mathrm{Sec}}
\newcommand{\Lie}{\mathcal{L}}
\DeclareRobustCommand{\tn}{\widetilde{\nabla}}
\DeclareRobustCommand{\tg}{\widetilde{g}}
\DeclareRobustCommand{\tN}{\widetilde{N}}
\newcommand{\Gam}{\Gamma}
\newcommand{\grad}{\operatorname{grad}}
\newcommand{\Hess}{\operatorname{Hess}}
\newcommand{\tr}{\operatorname{tr}}
\newcommand{\vt}{v^{\top}}
\newcommand{\vn}{v^{\perp}}
\newcommand{\dv}{\operatorname{div}}
\newcommand{\scal}{\operatorname{scal}}
\newcommand{\vol}{\operatorname{vol}}
\newcommand{\Tor}{\operatorname{Tor}}

\title{\textbf{Solitons on Statistical Manifolds and Their Submanifolds}}

\author{Aliya Naaz Siddiqui$^{1}$ and Bang-Yen Chen$^{2*}$\\
\small $^{1}$Division of Mathematics, School of Basic Sciences, Galgotias University,\\
\small Greater Noida 203201, Uttar Pradesh, India\\
\small E-mail: \ \texttt{aliya.siddiqui@galgotiasuniveristy.edu.in}\\
\small $^{2*}$Department of Mathematics, College of Natural Science, Michigan State University,\\
\small 619 Red Cedar Road, East Lansing, Michigan 48824, USA\\
\small E-mail: \  \texttt{chenb@msu.edu}}

\begin{document}
\maketitle

\begin{abstract}
Statistical manifolds, introduced by Amari in information geometry, carry a pair of dual torsion-free connections $(\nabla,\nabla^{*})$ together with a Riemannian metric. In this chapter we survey the theory of \emph{statistical solitons}, that is, vector fields and constants that solve the Ricci-soliton equation simultaneously for $\nabla$ and $\nabla^{*}$, together with the statistical analogue of almost quasi-Yamabe solitons. We discuss two principal settings: (i) statistical warped products $\R\times_f N_2$ and $N_1\times_f\R$  for (pseudo) Riemannian manifolds $N_{1}$ and $N_{2}$, where the soliton equation constrains the warping function; and (ii) statistical submanifolds of Kenmotsu statistical manifolds, where the position of the Reeb field $\xi$ relative to the submanifold and the presence of concircular or concurrent vector fields determine the Einstein, $\eta$-Einstein or quasi-Einstein character of the submanifold. We further review the results of Blaga and Chen on gradient solitons and statistical structures, in particular the connections $\nabla^{\theta}$ defined by a $1$-form, volume formulas and bounds for the Ricci tensor. We also describe Chen-type inequalities for statistical warped product submanifolds, which give the extrinsic context in which such solitons live. Key steps of proofs, remarks on the sharpness of hypotheses, worked examples, and open problems are included.
\medskip


\end{abstract}

\tableofcontents
\bigskip

\section{Introduction}

\subsection*{Information geometry and statistical manifolds}

Information geometry is an interdisciplinary field in which probability theory and statistics are studied with the tools of differential geometry. If $\Xi$ is a fixed event space and
\[
\overline{\mathbb B}=\{\mathcal P(x,\Theta)\ :\ \Theta=(\Theta_1,\dots,\Theta_n)\in\Omega\subset\R^n\}
\]
is a smooth $n$-parameter family of probability densities on $\Xi$, then the Fisher information
\[
g_{ij}(\Theta)=\int_\Xi \frac{\partial\log\mathcal P(x,\Theta)}{\partial\Theta_i}\,
\frac{\partial\log\mathcal P(x,\Theta)}{\partial\Theta_j}\ \mathcal P(x,\Theta)\,dx
\]
defines a Riemannian metric on $\overline{\mathbb B}$, so that points of the manifold are probability distributions. Together with the Fisher metric, $\overline{\mathbb B}$ carries a one-parameter family of affine connections (the $\alpha$-connections). Two connections $\nabla^{(1)}$ and $\nabla^{(-1)}$ are \emph{dual} with respect to the metric, and the abstract structure obtained in this way was formalised by Amari \cite{Amari1985} and Lauritzen \cite{Lauritzen1987} as a \emph{statistical manifold}. Affine (dually flat) geometry and the hyperbolic geometry of statistical models are closely related, which explains why curvature properties of statistical manifolds carry statistical meaning. Applications range from the ergodic hierarchy of statistical models on curved manifolds to the asymptotic behaviour of the information geometric complexity of Gaussian models in the presence of microcorrelations \cite{Ali2010,Gomez2017}.

From the viewpoint of differential geometry, a statistical manifold is a Riemannian manifold $(N,g)$ together with a torsion-free connection $\nabla$ such that $\nabla g$ is totally symmetric (the Codazzi condition); the dual connection $\nabla^{*}$ is then also torsion-free, and $\tfrac12(\nabla+\nabla^*)$ is the Levi-Civita connection. Because $\nabla$ is in general not metric, the usual Riemannian notions (curvature, sectional curvature, Ricci tensor, mean curvature of submanifolds) split into a pair of dual quantities, and one has to decide how to combine them. A convenient choice is the \emph{statistical curvature tensor} and the sectional curvature of Opozda \cite{Opozda2015,Opozda2016}. The submanifold theory of statistical manifolds, with dual second fundamental forms $h$, $h^*$ and dual shape operators, was developed by Vos \cite{Vos1989} and has since produced many curvature inequalities (Chen-type, Wintgen-type and Casorati-type) for statistical submanifolds \cite{MurathanSahin2018,Siddiqui2019}.    Additionally, statistical warped product submersions were introduced and investigated by Siddiqui et al. in \cite{Siddiqui2026}.

\subsection*{Ricci solitons and their variants}
A second source of inspiration is the theory of geometric flows. Hamilton's Ricci flow \cite{Hamilton1988} deforms a metric in the direction of its Ricci tensor, and \emph{Ricci solitons} are its self-similar solutions: a triple $(g,\zeta,\lambda)$ with
\[
\Ric+\tfrac12\Lie_\zeta g+\lambda g=0 .
\]
They generalise Einstein metrics (the case $\zeta=0$ or $\zeta$ Killing) and are classified as shrinking, steady or expanding according to the sign of $\lambda$; they model the formation of singularities of the flow and are of interest in geometric analysis, general relativity and string theory. The \emph{Yamabe flow} instead preserves the conformal class of the metric and its fixed points are metrics of constant scalar curvature; Yamabe solitons are the corresponding self-similar solutions. For $\dim=2$ the two flows coincide, while in higher dimensions they differ because the Ricci flow does not preserve the conformal class. Further generalisations, in which the soliton constant is replaced by a function and an extra term $\omega\,\eta\otimes\eta$ is included, lead to \emph{almost} and \emph{quasi}-Yamabe solitons \cite{ChenDeshmukh2018,Jub2020}, $\eta$-Ricci solitons and Ricci almost solitons.

These notions have been studied intensively on almost contact manifolds, and in particular on Kenmotsu manifolds \cite{Kenmotsu1972}, see, for instance, \cite{NagarajaPremalatha2012,Blaga2015,CalinCrasmareanu2010}. Another fruitful direction is to study solitons on \emph{submanifolds}. Here, concircular vector fields \cite{Fialkow1939}, whose integral curves are geodesics and for which $\nabla_Ev=\delta E$, play a distinguished role, as shown by Chen \cite{Chen2015} for Ricci solitons on Riemannian submanifolds: the tangential part of a concircular field of the ambient space is again concircular on a submanifold precisely under an umbilicity condition, and the soliton equation then expresses the Ricci tensor of the submanifold in terms of its second fundamental form.

\subsection*{Solitons in the statistical setting}

The soliton equation involves the Ricci tensor and the covariant derivative of the potential field, and both make sense for an arbitrary linear connection. This observation of Crasmareanu \cite{Crasmareanu2018} allows one to speak about a $\nabla$-Ricci soliton for a connection $\nabla$ different from the Levi-Civita connection. On a statistical manifold there are two natural candidates, $\nabla$ and $\nabla^*$, and demanding that the same pair $(\zeta,\lambda)$ solves the equation for both leads to the notion of a \emph{statistical soliton}. The condition is genuinely stronger than a soliton condition for a single connection, and it couples the geometry of the two dual connections. As a consequence, existence questions become rigid: obstructions appear, for instance, from the difference tensor $\mathcal T=\nabla^0-\nabla^*$ of the statistical structure (see Section~\ref{sec:xi}).  In addition, several necessary and sufficient conditions for some particular couples  $(g,\nabla)$ of pseudo-Riemannian metrics and affine connections to be statistical structures were obtained if one has gradient almost Einstein, almost Ricci, almost Yamabe solitons, or a more general type of solitons on the manifold \cite{BlagaChen2021}.

\subsection*{Warped products and Kenmotsu statistical manifolds}

Two families of manifolds provide the natural testing ground for this theory. The first consists of \emph{warped products} \cite{BishopONeill1969}. They are among the most fruitful generalisations of Riemannian products, are the standard building blocks of space-time models in general relativity \cite{Chen2011,Chen2017}, and carry a natural dualistic structure \cite{Todjihounde2006}; Einstein statistical warped products and Wintgen-type inequalities for their submanifolds have been studied in \cite{AytimurOzgur2018,MurathanSahin2018}. Since the soliton equation on a warped product $N_1\times_fN_2$ decomposes into base, fibre, and mixed components, it produces ordinary differential conditions on the warping function $f$, which makes solitons on warped products computable and geometrically transparent (Section~\ref{sec:warpsol}).

The second family consists of \emph{Kenmotsu statistical manifolds}, introduced by Furuhata, Hasegawa, Okuyama and Sato \cite{Furuhata2017} by imposing a compatible statistical structure on a Kenmotsu manifold; they are obtained as warped products of a holomorphic statistical manifold \cite{Furuhata2009} and a line. Kenmotsu geometry has applications in the mechanics of time-dependent Hamiltonian systems, geometrical optics, and thermodynamics, and its statistical counterpart is of comparable importance. The Reeb field $\xi$ gives a distinguished potential for the soliton equation, and the relative position of $\xi$ with respect to a submanifold (tangent or normal) leads to different classes of submanifolds: $\eta$-Einstein in the tangent case and Einstein in the normal case \cite{Ahmadini2024}.

\subsection*{Aim and contents of this chapter}

The purpose of this chapter is to give a unified and reasonably self-contained account of these developments. We shall explain how
\begin{itemize}[itemsep=1pt]
\item the statistical soliton equation is formulated and why it is rigid (Section~\ref{sec:solitons});
\item gradient solitons decide when Hessian or Ricci metrics, together with a suitable connection, form (nearly) statistical structures, and the connection $\nabla^{\theta}$ defined by a $1$-form leads to volume formulas and Ricci-norm bounds (Section~\ref{sec:gradient});
\item on statistical warped products $\R\times_fN_2$ and $N_1\times_f\R$ the soliton equation constrains the warping function and induces solitons on base or fibre (Section~\ref{sec:warpsol});
\item on submanifolds of Kenmotsu statistical manifolds, the position of $\xi$ and the presence of concircular or concurrent vector fields determine the Einstein, $\eta$-Einstein or quasi-Einstein character of the submanifold, and the scalar curvature relation for almost quasi-Yamabe solitons (Sections~\ref{sec:kenmotsu}--\ref{sec:yamabe});
\item Chen-type inequalities for statistical warped product submanifolds relate the intrinsic scalar curvature to the extrinsic mean curvatures and to the Laplacian of the warping function (Section~\ref{sec:chen}).
\end{itemize}
We include the key steps of the proofs, remarks on the role of the hypotheses and on sign conventions, explicit examples (Section~\ref{sec:examples}), and a list of open problems (Section~\ref{sec:conclusion}).

\paragraph{Organisation.}
Section~\ref{sec:prelim} collects the basic notions on statistical manifolds, statistical curvature and statistical submanifolds. Section~\ref{sec:solitons} introduces statistical solitons, almost quasi-Yamabe solitons and concircular vector fields. Section~\ref{sec:gradient} reviews gradient solitons and statistical structures after Blaga and Chen. Section~\ref{sec:warped} recalls statistical warped products and Section~\ref{sec:warpsol} describes statistical solitons on them. Section~\ref{sec:kenmotsu} reviews Kenmotsu statistical manifolds, and Sections~\ref{sec:xi}--\ref{sec:yamabe} study solitons on their submanifolds, with and without concircular fields. Section~\ref{sec:chen} describes the Chen inequality for statistical warped product submanifolds. Section~\ref{sec:examples} contains explicit examples, and Section~\ref{sec:conclusion} concludes with remarks and open problems.

\paragraph{Conventions.}
All manifolds are smooth and connected; $\Gam(TN)$ denotes smooth vector fields; $\nabla^0$ is the Levi-Civita connection; objects with a star refer to the dual connection $\nabla^*$, and a superscript $0$ to the Levi-Civita connection. Statements proved for $\nabla$ have dual versions for $\nabla^*$ obtained by interchanging $\nabla$ and $\nabla^*$ (and $\mathcal T\mapsto-\mathcal T$).

\section{Preliminaries}\label{sec:prelim}

\subsection{Statistical manifolds}

Let $(\tN,\tg)$ be a (pseudo-)Riemannian manifold and $\tn$ an affine connection on $\tN$. The \emph{dual connection} $\tn^{*}$ of $\tn$ with respect to $\tg$ is defined by
\begin{equation}\label{eq:dual}
E\,\tg(F,G)=\tg(\tn_E F,G)+\tg(F,\tn^{*}_E G),\qquad E,F,G\in\Gam(T\tN),
\end{equation}
where $\Gam(T\tN)$ denotes the space of smooth sections on the tangent bundle $T\tN$ of $\tN$.

\begin{definition}
The triple $(\tN,\tn,\tg)$ is a \emph{statistical manifold} if 
\begin{enumerate}
    \item the Codazzi equation $$(\tn_E\tg)(F,G)=(\tn_F\tg)(E,G)$$ holds, and
    \item $\tn$ is torsion-free.
\end{enumerate}
In this case $(\tn^{*},\tg)$ is again a statistical structure, $$(\tn^{*})^{*}=\tn$$ and $$\tn^{0}=\tfrac12(\tn+\tn^{*})$$ is the Levi-Civita connection of $\tg$.
\end{definition}

Equivalently, a statistical structure is obtained as $$\tn=\tn^{0}+\mathcal T,$$ where $\mathcal T\in\Gam(T^{(1,2)}\tN)$ is the symmetric \emph{difference tensor} satisfying
\begin{equation}\label{eq:T}
\mathcal T(E,F)=\mathcal T(F,E),\qquad \tg(\mathcal T(E,F),G)=\tg(F,\mathcal T(E,G)),
\end{equation}
and then $\tn^{*}=\tn^{0}-\mathcal T$. 

\subsection{Statistical curvature}

Let $\Rie$, $\Rie^{*}$ and $\Rie^{0}$ denote the curvature tensors of $\nabla$, $\nabla^{*}$ and $\nabla^{0}$. The \emph{statistical curvature tensor} is
\begin{equation}\label{eq:S}
S=\tfrac12(\Rie+\Rie^{*})=\Rie^{0}+[\mathcal T_E,\mathcal T_F],
\end{equation}
(see \cite{Furuhata2017}). Following Opozda \cite{Opozda2015,Opozda2016}, the \emph{statistical sectional curvature} of the plane spanned by orthonormal $E,F$ is
\begin{equation}\label{eq:Sec}
\Sec(E\wedge F)=g(S(E,F)F,E)=\tfrac12\big(g(\Rie(E,F)F,E)+g(\Rie^{*}(E,F)F,E)\big),
\end{equation}
and for an orthonormal basis $\{e_1,\dots,e_m\}$ the scalar curvature is $$\sigma^{\nabla,\nabla^*}=\sum_{i<j}\Sec(e_i\wedge e_j).$$ The Ricci tensors $\Ric$ and $\Ric^{*}$ of $\nabla$ and $\nabla^{*}$ are the traces of $E\mapsto\Rie(E,F)G$ and $E\mapsto \Rie^*(E,F)G$.

A statistical manifold $(\tN,\tn,\tg)$ is of \emph{constant curvature} $\tilde c$ if $$S(E,F)G=\tilde c\,(g(F,G)E-g(E,G)F).$$

\subsection{Statistical submanifolds}

Let $N$ be a submanifold of $(\tN,\tn,\tg)$ with induced statistical structure $(\nabla,g)$. The Gauss and Weingarten formulae read \cite{Vos1989}
\begin{align}
\tn_EF&=\nabla_EF+h(E,F), & \tn^{*}_EF&=\nabla^{*}_EF+h^{*}(E,F),\label{eq:gauss}\\
\tn_E U&=-A_UE+D^{\perp}_EU, & \tn^{*}_E U&=-A^{*}_UE+D^{\perp *}_EU,\label{eq:weingarten}
\end{align}
for $E,F\in\Gam(TN)$, $U\in\Gam(T^{\perp}N)$, where $h,h^{*}$ are the symmetric imbedding curvature tensors, related to the shape operators by
\begin{equation}\label{eq:AhA}
\tg(h(E,F),U)=g(A^{*}_UE,F),\qquad \tg(h^{*}(E,F),U)=g(A_UE,F).
\end{equation}
The Gauss equation takes the form
\begin{equation}\label{eq:GaussEq}
\tg(\widetilde\Rie(E,F)G,H)=g(\Rie(E,F)G,H)+\tg(h(E,G),h^{*}(F,H))-\tg(h^{*}(E,H),h(F,G)),
\end{equation}
together with its dual. The mean curvature vectors are $$H=\frac1m\tr h \quad {\rm and}\quad H^{*}=\frac1m\tr h^{*},$$ and $N$ is called \emph{minimal} if $H=H^{*}=0$, \emph{totally geodesic} if $h=h^*=0$, and \emph{totally umbilical} if $h(E,F)=g(E,F)H$, and $h^{*}(E,F)=g(E,F)H^{*}$, respectively. Further, we say $N$ is \emph{$U$-umbilical} for a normal field $U$ if we have $A_U=fI$ and $A^{*}_U=fI$ for some function $f$.

Note that $A_U^{0}:=\tfrac12(A_U+A_U^{*})$ is the shape operator of $N$ for the Levi-Civita connection, and $h^0=\frac12(h+h^*)$.

\section{Statistical Solitons, Yamabe-Type Solitons and Concircular Fields}\label{sec:solitons}

\subsection{Ricci solitons for general connections}

A Ricci soliton $(g,\zeta,\lambda)$ on a Riemannian manifold is classified as \emph{shrinking}, \emph{steady} or \emph{expanding} according to whether $\lambda<0$, $\lambda=0$ or $\lambda>0$ in the convention 
$\Ric+\frac12\Lie_\zeta g+\lambda g=0$. Crasmareanu \cite{Crasmareanu2018} observed that for a manifold with an arbitrary linear connection $\tn$, one may write the soliton condition as
\[
\tn\zeta+\widetilde Q+\lambda I=0,\qquad \tg(\widetilde QE,F)=\widetilde\Ric(E,F),
\]
where $\widetilde Q$ is the Ricci operator and $(\tn\zeta)(E)=\tn_E\zeta$. A statistical manifold is called \emph{Ricci-symmetric} if $\widetilde Q$ (equivalently $\widetilde Q^{*}$) is symmetric.

\begin{definition}[Statistical soliton]\label{def:statsol}
A pair $(\zeta,\lambda)$ is a \emph{statistical soliton} on a Ricci-symmetric statistical manifold $(\tN,\tn,\tg)$ if $(\tg,\zeta,\lambda)$ is simultaneously a $\tn$-Ricci and a $\tn^{*}$-Ricci soliton:
\begin{equation}\label{eq:statsol}
\tn\zeta+\widetilde Q+\lambda I=0,\qquad \tn^{*}\zeta+\widetilde Q^{*}+\lambda I=0 .
\end{equation}
\end{definition}
In particular, if $\tn=\tn^*=\tn^0$ one recovers the classical gradient-free Ricci soliton equation. In terms of tensors, \eqref{eq:statsol} reads
\[
g(\nabla_E\zeta,F)+\Ric(E,F)+\lambda g(E,F)=0,\qquad
g(\nabla^*_E\zeta,F)+\Ric^*(E,F)+\lambda g(E,F)=0 .
\]
Adding the two equations and using that $\nabla^{0}\!=\!\frac12(\nabla+\nabla^*)$ gives the Levi-Civita-type equation $\frac12(\Lie_\zeta g)+\frac12(\Ric+\Ric^*)+\lambda g=0$, since $\Lie_\zeta g(E,F)=g(\nabla_E\zeta,F)+g(E,\nabla^{*}_F\zeta)$ for torsion-free dual connections.

\subsection{Almost quasi-Yamabe solitons}

The Yamabe flow deforms the metric within its conformal class, and its fixed points have constant scalar curvature; a Yamabe soliton satisfies $$\Lie_\zeta g=2(R-\lambda)g.$$ Following Chen and Deshmukh \cite{ChenDeshmukh2018} (see also \cite{Jub2020}), when $\lambda$ is promoted to a function and a term $\omega\,\eta\otimes\eta$ is added we obtain:

\begin{definition}[Almost quasi-Yamabe soliton]
Let $(B,g)$ be a Riemannian manifold of dimension $>2$, $\zeta$ a vector field, $\eta$ a $1$-form and $\lambda,\omega$ smooth functions. The data $(g,\zeta,\lambda,\omega)$ define an \emph{almost quasi-Yamabe soliton} if
\begin{equation}\label{eq:aqy}
\tfrac12\Lie_\zeta g+(\lambda-R)\,g+\omega\,\eta\otimes\eta=0,
\end{equation}
where $R$ is the scalar curvature.
\end{definition}

\subsection{Concircular and concurrent vector fields}
A vector field $v$ on a Riemannian manifold is \emph{concircular} \cite{Fialkow1939} if $\nabla_Ev=\delta E$ for a non-trivial function $\delta$; it is \emph{concurrent} if $\delta=1$. On a statistical manifold we say that $v$ is concircular if
\begin{equation}\label{eq:concirc}
\tn_Ev=\delta E,\qquad \tn^{*}_Ev=\delta E .
\end{equation}
Chen \cite{Chen2015} used concircular fields to investigate Ricci solitons on Riemannian submanifolds; the statistical version is developed in Section~\ref{sec:concirc}.

\section{Gradient Solitons and Statistical Structures}\label{sec:gradient}
The soliton equation can also be read in the opposite direction. Instead of asking which solitons live on a given statistical manifold, one asks for which pairs $(h,\nabla)$ constructed from a soliton the Codazzi condition holds, that is, when the soliton produces a (nearly, quasi-)statistical structure. Blaga and Chen developed this point of view \cite{BlagaChen2021} (see also \cite{Blaga2019}); in this section we review their main results, in the notation of the present chapter.

Throughout, $(M,g)$ is an $n$-dimensional pseudo-Riemannian manifold with Levi-Civita connection $\nabla^{0}$, Ricci tensor $\Ric^{0}$, Ricci operator $Q^{0}$ ($g(Q^0X,Y)=\Ric^0(X,Y)$) and scalar curvature $\scal^{0}$. For a function $f$ we write $\Hess^{0}f(X,Y)=g(\nabla^0_X\grad f,Y)$ and $\Delta_0 f=\tr_g\Hess^0 f=\dv(\grad f)$.

\indent 
\emph{Note that this Laplacian carries the sign $+\tr\Hess$, opposite to the convention used in Section~\ref{sec:chen}.} 

Solitons are written here in the form of \cite{BlagaChen2021},
\begin{equation}\label{eq:bc-soliton}
\nabla\zeta+J=\lambda I,
\end{equation}
so that compared with \eqref{eq:statsol} the soliton function is replaced by its negative; the vector field is called \emph{gradient} if $\zeta=\grad f$. The Reeb field of Section~\ref{sec:kenmotsu} keeps the letter $\xi$, and the dual $1$-form of $\zeta$ is denoted $\theta=g(\zeta,\cdot)$.

\subsection{Nearly statistical and quasi-statistical structures}
Let $\nabla$ be an affine connection with torsion $\Tor^{\nabla}$ and $h$ a $(0,2)$-tensor field. Put
\[
(d^{\nabla}h)(X,Y,Z)=(\nabla_Xh)(Y,Z)-(\nabla_Yh)(X,Z)+h(\Tor^{\nabla}(X,Y),Z).
\]
The pair $(h,\nabla)$ is a \emph{quasi-statistical structure} if $d^\nabla h=0$ (Matsuzoe \cite{Matsuzoe2013}), and a \emph{nearly statistical structure} if $h$ is symmetric, $\nabla$ is torsion-free and $d^\nabla h=0$, that is, $(\nabla_Xh)(Y,Z)=(\nabla_Yh)(X,Z)$. A statistical structure in the sense of Section~\ref{sec:prelim} is a nearly statistical structure $(g,\nabla)$ with $g$ non-degenerate. If moreover $\Rie^\nabla=0$, then $\Rie^{\nabla^*}=0$ as well, and one obtains a \emph{Hessian} (dually flat) structure \cite{Shima1977}.

\subsection{Hessian metrics and Ricci metrics}
If $\Hess^0f$ is non-degenerate (of constant signature) it is itself a pseudo-Riemannian metric, and so is $\Ric^0$ when non-degenerate; their Levi-Civita connections are expressed through Koszul's formula, e.g.
\[
2g\big(\nabla^{\Ric^0}_XY-\nabla^0_XY,\,Q^0Z\big)=g((\nabla^0_YQ^0)X,Z)+g((\nabla^0_XQ^0)Y,Z)-g((\nabla^0_ZQ^0)X,Y).
\]
From this one reads off the following facts \cite{Blaga2019}:
\begin{enumerate}[label=(\roman*),itemsep=0pt]
\item $(\Ric^0,\nabla^0)$ is a statistical structure if  and only if $(\nabla^0_XQ^0)Y=(\nabla^0_YQ^0)X$ for all $X,Y$;
\item if $(g,\nabla^{\Ric^0})$ is a statistical structure and $\mathcal D:=\nabla^{\Ric^0}-\nabla^0$, then $$g(X,\mathcal D(Y,Z))=g(Y,\mathcal D(X,Z));$$
\item if $(g,\nabla)$ is quasi-statistical, then $(\Ric^0,\nabla)$ is quasi-statistical if  and only if $$g((\nabla_XQ^0)Z,Y)=g((\nabla_YQ^0)Z,X).$$
\end{enumerate}

\begin{theorem}[\cite{BlagaChen2021}]\label{thm:bc-hess0}
$(\Hess^0f,\nabla^0)$ is a statistical structure on $M$ if and only if the radial curvature of $f$ vanishes, i.e.\ $\Rie^{0}(X,Y)\grad f=0$ for all $X,Y\in\Gam(TM)$.
\end{theorem}
\begin{proof}
Since $\nabla^0$ is torsion-free and $\nabla^0g=0$, expanding $$(\nabla^0_X\Hess^0f)(Y,Z)-(\nabla^0_Y\Hess^0f)(X,Z)$$ gives
$$g\big((\nabla^{0})^2_{X,Y}\grad f-(\nabla^{0})^2_{Y,X}\grad f,Z\big)=g(\Rie^0(X,Y)\grad f,Z),$$ where $\nabla^2_{X,Y}=\nabla_X\nabla_Y-\nabla_{\nabla_XY}.$
\end{proof}

\begin{remark} The term \emph{radial curvature} goes back to Klingenberg \cite{Klingenberg1963}.
\end{remark}

For a general statistical structure, the same computation is slightly less rigid. Let $(g,\nabla)$ be statistical, $C=\nabla g$ (totally symmetric), and put $\Hess^\nabla f(X,Y)=g(\nabla_X\grad f,Y)$; this is symmetric because $\nabla=\nabla^0+\mathcal T$ with $g(\mathcal T(\cdot,\cdot),\cdot)$ totally symmetric, cf.\ \eqref{eq:T}.

\begin{theorem}\label{thm:bc-hessnabla}
Let $(g,\nabla)$ be a statistical structure. Then $(\Hess^\nabla f,\nabla)$ is a statistical structure if  and only if
\[
g(\Rie^\nabla(X,Y)\grad f,Z)=C(Y,\nabla_X\grad f,Z)-C(X,\nabla_Y\grad f,Z),\qquad X,Y,Z\in\Gam(TM).
\]
In particular, if $\nabla g=0$ (e.g.\ $\nabla=\nabla^0$) the condition reduces to $\Rie^\nabla(X,Y)\grad f=0$.
\end{theorem}
\begin{proof}
With $H=\Hess^\nabla f$ one computes $$(\nabla_XH)(Y,Z)=C(X,\nabla_Y\grad f,Z)+g(\nabla^2_{X,Y}\grad f,Z).$$ Antisymmetrising in $X,Y$ and using $\Tor^\nabla=0$ gives $$(d^\nabla H)(X,Y,Z)=C(X,\nabla_Y\grad f,Z)-C(Y,\nabla_X\grad f,Z)+g(\Rie^\nabla(X,Y)\grad f,Z).$$
\end{proof}

\begin{remark}
In \cite{BlagaChen2021} the statement is given as ``$\Rie^\nabla(X,Y)\grad f=0$'' for every statistical $\nabla$; the two $C$-terms above are what remains of the Codazzi condition for $\nabla g$ and vanish when $\nabla g=0$, which is the case used in the applications below.
\end{remark}

\subsection{Gradient almost solitons}
Let $f,\lambda$ be smooth functions. The triple $(g,f,\lambda)$ is a
\begin{itemize}[itemsep=0pt]
\item \emph{gradient almost Ricci soliton} if $\Hess^0f+\Ric^0=\lambda g$;
\item \emph{gradient almost Einstein soliton} if $\Hess^0f+\Ric^0=\big(\lambda+\tfrac12\scal^0\big)g$;
\item \emph{gradient almost Yamabe soliton} if $\Hess^0f=(\lambda-\scal^0)g$.
\end{itemize}
When $\lambda$ is constant, ``almost'' is dropped. Differentiating covariantly gives, respectively,
$$(\nabla^0_X\Hess^0f)(Y,Z)+(\nabla^0_X\Ric^0)(Y,Z)=X(\lambda)g(Y,Z),$$ with $\lambda$ replaced by $\lambda+\tfrac12\scal^0$ for Einstein solitons, and $$(\nabla^0_X\Hess^0f)(Y,Z)=X(\lambda-\scal^0)g(Y,Z)$$ for Yamabe solitons. This yields:

\begin{proposition}[\cite{BlagaChen2021}]\label{prop:bc-const}
\begin{enumerate}[label=(\roman*),itemsep=0pt]
\item If $(g,f,\lambda)$ is a gradient Ricci soliton, then $(\Hess^0f,\nabla^0)$ is statistical if and only if $(\Ric^0,\nabla^0)$ is statistical.
\item If $(g,f,\lambda)$ is a gradient Einstein soliton and $\scal^0$ is constant, the same equivalence holds.
\item If $(g,f,\lambda)$ is a gradient Yamabe soliton and $\scal^0$ is constant, then $(\Hess^0f,\nabla^0)$ is statistical.
\end{enumerate}
\end{proposition}

\begin{proposition}[\cite{BlagaChen2021}]\label{prop:bc-almost}
Let $\Ric^0$ be non-degenerate.
\begin{enumerate}[label=(\roman*),itemsep=0pt]
\item If $(g,f,\lambda)$ is a gradient almost Ricci soliton, then $(\Ric^0,\nabla^0)$ is statistical if and only if
$$g(\Rie^0(X,Y)\grad f,Z)=X(\lambda)g(Y,Z)-Y(\lambda)g(X,Z).$$
\item If $(g,f,\lambda)$ is a gradient almost Einstein soliton, then $(\Ric^0,\nabla^0)$ is statistical if and only if
$$g(\Rie^0(X,Y)\grad f,Z)=X\big(\lambda+\tfrac{\scal^0}2\big)g(Y,Z)-Y\big(\lambda+\tfrac{\scal^0}2\big)g(X,Z).$$
\end{enumerate}
\end{proposition}
These follow by combining the differentiated soliton equations with the identity $$d^{\nabla^0}\Hess^0f(X,Y,Z)=g(\Rie^0(X,Y)\grad f,Z)$$ from the proof of Theorem~\ref{thm:bc-hess0}. Analogous conditions for a statistical connection $\nabla\neq\nabla^0$ can be written down in the same way \cite[Prop.~2.6]{BlagaChen2021}.

\subsection{Almost $(\nabla,J,\zeta,\lambda)$-solitons}
The three soliton equations above are of the form $\nabla^0\zeta+J=\lambda' I$ with $\zeta=\grad f$, where $J=Q^0$ and $\lambda'=\lambda$ (Ricci), $J=Q^0$ and $\lambda'=\lambda+\frac12\scal^0$ (Einstein), $J=0$ and $\lambda'=\lambda-\scal^0$ (Yamabe). This suggests the general notion: an \emph{almost $(\nabla,J,\zeta,\lambda)$-soliton} is a datum consisting of an affine connection $\nabla$, a $(1,1)$-tensor $J$, a vector field $\zeta$ and a function $\lambda$ satisfying \eqref{eq:bc-soliton}. Put $\Omega=g(J\cdot,\cdot)$.

\begin{lemma}\label{lem:bc-sym}
For an almost $(\nabla,J,\zeta,\lambda)$-soliton, $\Omega$ is symmetric if and only if $\nabla\zeta$ is $g$-self-adjoint: $g(\nabla_X\zeta,Y)=g(X,\nabla_Y\zeta)$.
\end{lemma}
\begin{lemma}[\cite{Blaga2019}]\label{lem:bc-codazzi}
For torsion-free $\nabla$: $(\nabla_X\Omega)(Y,Z)=(\nabla_Y\Omega)(X,Z)$ if and only if $(\nabla_Xg)(JY,Z)-(\nabla_Yg)(JX,Z)=g((\nabla_YJ)X-(\nabla_XJ)Y,Z)$. In particular $\Omega$ is $\nabla^0$-Codazzi if and only if $J$ is $\nabla^0$-Codazzi.
\end{lemma}
\begin{remark}
If $\Omega$ is $\nabla^0$-Codazzi and $J$ is a Killing tensor ($(\nabla^0_XJ)X=0$), then polarising gives $\nabla^0J=0$.
\end{remark}

\begin{proposition}[\cite{BlagaChen2021}]\label{prop:bc-nearly}
Let $(\nabla,J,\zeta,\lambda)$ be an almost soliton on $(M,g)$ with $\Omega$ symmetric and $\nabla$ torsion-free. Then $(\Omega,\nabla)$ is nearly statistical if and only if
\[
g(\Rie^\nabla(X,Y)\zeta,Z)=(\nabla_Xg)(JY,Z)-(\nabla_Yg)(JX,Z)+g(X(\lambda)Y-Y(\lambda)X,Z).
\]
\end{proposition}
\begin{proof}
From \eqref{eq:bc-soliton}, $\Rie^\nabla(X,Y)\zeta=(\nabla_YJ)X-(\nabla_XJ)Y+X(\lambda)Y-Y(\lambda)X$; now apply Lemma~\ref{lem:bc-codazzi}.
\end{proof}

\begin{corollary}\label{cor:bc-LC}
For $\nabla=\nabla^0$ and $\Omega$ symmetric, $(\Omega,\nabla^0)$ is nearly statistical if and only if $\Rie^0(X,Y)\zeta=X(\lambda)Y-Y(\lambda)X$. If $\lambda$ is constant and $\zeta=\grad f$, the following are equivalent: (i) $(\Omega,\nabla^0)$ is nearly statistical; (ii) $\Rie^0(X,Y)\grad f=0$; (iii) $(\Hess^0f,\nabla^0)$ is nearly statistical.
\end{corollary}

\subsection{Connections defined by $1$-forms}\label{subsec:bc-forms}
Inspired by projective equivalence, for an arbitrary $1$-form $\theta=g(\zeta,\cdot)$ on $(M,g)$ consider
\begin{equation}\label{eq:nablatheta}
\nabla^{\theta}:=\nabla^0+\theta\otimes I+I\otimes\theta+g\otimes\zeta,\quad\text{i.e.}\quad
\nabla^\theta_XY=\nabla^0_XY+\theta(X)Y+\theta(Y)X+g(X,Y)\zeta .
\end{equation}
Then $\Tor^{\nabla^\theta}=0$ and $(\nabla^\theta_Xg)(Y,Z)=-2\{\theta(X)g(Y,Z)+\theta(Y)g(X,Z)+\theta(Z)g(X,Y)\}$ is totally symmetric, so $d^{\nabla^\theta}g=0$.

\begin{proposition}[\cite{BlagaChen2021}]\label{prop:bc-theta}
For every $1$-form $\theta$, $(g,\nabla^\theta)$ is a statistical structure and $\nabla^{-\theta}$ is its dual connection. Moreover $\nabla^\theta_\zeta\zeta=\nabla^0_\zeta\zeta+3|\zeta|^2\zeta$, and:
\begin{enumerate}[label=(\roman*),itemsep=0pt]
\item $\zeta$ is $\nabla^\theta$-parallel if and only if $(\nabla^0,\,J=2\theta\otimes\zeta,\,\zeta,\,\lambda=-|\zeta|^2)$ is an almost soliton;
\item $\theta$ is $\nabla^\theta$-parallel if and only if $(\nabla^0,\,J=-2\theta\otimes\zeta,\,\zeta,\,\lambda=|\zeta|^2)$ is an almost soliton.
\end{enumerate}
\end{proposition}
\begin{proof}
By direct calculation, $$g(\nabla^\theta_XY,Z)+g(Y,\nabla^{-\theta}_XZ)=Xg(Y,Z).$$ For (i), $\nabla^\theta_X\zeta=\nabla^0_X\zeta+2\theta(X)\zeta+|\zeta|^2X$;\\
\noindent 
For (ii), $(\nabla^\theta_X\theta)(Y)=(\nabla^0_X\theta)(Y)-2\theta(X)\theta(Y)-|\zeta|^2g(X,Y)$.
\end{proof}
\begin{remark}
The signs of $\lambda$ in (i) and (ii) are interchanged in the printed version of \cite{BlagaChen2021}; the signs above are the ones obtained from the two displayed computations.
\end{remark}

The curvature of $\nabla^\theta$ is
\begin{align*}
(\Rie^{\theta}-\Rie^{0})(X,Y)Z&=[g(Y,\nabla^0_X\zeta)-g(X,\nabla^0_Y\zeta)]Z\\
&\quad+[\theta(Y)\theta(Z)+g(Y,Z)|\zeta|^2-g(Z,\nabla^0_Y\zeta)]X\\
&\quad-[\theta(X)\theta(Z)+g(X,Z)|\zeta|^2-g(Z,\nabla^0_X\zeta)]Y\\
&\quad+g(Y,Z)\nabla^0_X\zeta-g(X,Z)\nabla^0_Y\zeta\\
&\quad+[\theta(X)g(Y,Z)-\theta(Y)g(X,Z)]\zeta .
\end{align*}
In particular, if $\zeta$ is $g$-null and $\nabla^0$-parallel (the situation on Walker manifolds \cite{Vazquez2009}), then $(\Rie^\theta-\Rie^0)(X,Y)Z\in\ker\theta$.

\begin{example}\label{ex:bc-kenmotsu}
Let $(M,\varphi,\xi,\eta,g)$ be a Kenmotsu manifold and $\nabla^\eta=\nabla^0+\eta\otimes I+I\otimes\eta+g\otimes\xi$. Since $\nabla^0\xi=I-\eta\otimes\xi$, we get $\nabla^\eta\xi=2I+\eta\otimes\xi$, so $(\nabla^\eta,J=-\eta\otimes\xi,\xi,\lambda=2)$ is an almost soliton on $M$.
\end{example}

\paragraph{Divergence, Laplacian and Hessian.}
Tracing \eqref{eq:nablatheta} gives $\dv^{\theta}Z=\dv^0Z+(n+2)\theta(Z)$ and hence, with $\Delta_\theta=\dv^\theta\circ\grad$,
\begin{align*}
\Delta_\theta\tilde f&=\Delta_0\tilde f+(n+2)\theta(\grad\tilde f),\\
\Hess^\theta\tilde f(X,Y)&=\Hess^0\tilde f(X,Y)+\theta(\grad\tilde f)g(X,Y)+\theta(X)d\tilde f(Y)+\theta(Y)d\tilde f(X).
\end{align*}
For $\zeta=\grad f$ (so $\theta=df$) one gets
\begin{equation}\label{eq:bc-Hessdf}
\Hess^{df}f=\Hess^0f+|\grad f|^2g+2\,df\otimes df,\qquad \Delta_{df}f=\Delta_0f+(n+2)|\grad f|^2 .
\end{equation}
Consequently, if $f$ is $\Delta_0$-harmonic then $\Delta_{df}f\ge0$ whenever $|\grad f|^2\ge 0$; and if $f$ is $\Delta_{df}$-harmonic on a compact Riemannian manifold then $f$ is locally constant. Moreover, a $\Delta_{df}$-harmonic function $\tilde f$ is also $\Delta_0$-harmonic if and only if $\grad f\perp\grad\tilde f$. Integrating $\dv^{df}$ over a compact $M$ gives $$\int_M\dv^{df}X\,d\mu_g=(n+2)\int_Mg(\grad f,X)\,d\mu_g.$$

\paragraph{Ricci tensor.}
Tracing the curvature formula yields
\[
\Ric^\theta(Y,Z)=\Ric^0(Y,Z)+g(Y,Z)\{n|\zeta|^2+\dv^0\zeta\}+(n-2)\theta(Y)\theta(Z)+g(Y,\nabla^0_Z\zeta)-(n+1)g(Z,\nabla^0_Y\zeta).
\]
\begin{proposition}\label{prop:bc-equiaffine}
$\nabla^\theta$ is equiaffine if and only if $\nabla^0\zeta$ is $g$-self-adjoint, i.e.\ $\theta$ is closed. In particular $\nabla^{df}$ is always equiaffine.
\end{proposition}
Indeed, an equiaffine torsion-free connection is characterised by symmetry of its Ricci tensor \cite{NomizuSasaki1994}, and the only non-symmetric part of $\Ric^\theta$ is $g(Y,\nabla^0_Z\zeta)-(n+1)g(Z,\nabla^0_Y\zeta)$, whose antisymmetric part is $(n+2)$ times that of $g(Y,\nabla^0_Z\zeta)$.

\begin{remark}[Conjugate Ricci symmetry]
Recall that a statistical manifold is \emph{conjugate Ricci-symmetric} if $\Ric^\nabla=\Ric^{\nabla^*}$ \cite{Min2015}. For $(g,\nabla^\theta,\nabla^{-\theta})$ the difference $\Ric^\theta-\Ric^{-\theta}$ is the part of the Ricci formula odd in $\zeta$, namely $2\{g(Y,Z)\dv^0\zeta+g(Y,\nabla^0_Z\zeta)-(n+1)g(Z,\nabla^0_Y\zeta)\}$. Its antisymmetric part is $(n+2)\,d\theta$ and its symmetric part is $2\dv^0\zeta\,g-n\Lie_\zeta g$. Hence $(M,g,\nabla^\theta,\nabla^{-\theta})$ is conjugate Ricci-symmetric if and only if $\theta$ is closed and $\Lie_\zeta g=\tfrac2n(\dv^0\zeta)g$, i.e.\ $\zeta$ is a closed conformal vector field (for $\zeta=\grad f$: $\Hess^0f=\tfrac1n(\Delta_0f)g$); this is a condition, not an automatic property of the pair.
\end{remark}

In the gradient case $\zeta=\grad f$, $\theta=df$, one has
\begin{equation}\label{eq:bc-Ricdf}
\Ric^{df}=\Ric^0+\{n|\grad f|^2+\Delta_0f\}g+(n-2)\,df\otimes df-n\Hess^0f,
\end{equation}
and by tracing, $$\scal^{df}=\scal^0+(n-1)(n+2)|\grad f|^2;$$ hence $\scal^{df}\ge\scal^0$ when $|\grad f|^2\ge0$. For the Ricci operators, $$Q^{df}=Q^0+\{n|\grad f|^2+\Delta_0f\}I+(n-2)\,df\otimes\grad f-n\nabla^0\grad f.$$

\begin{proposition}\label{prop:bc-equivalences}
Let $\zeta=\grad f$, $\theta=df$. Then:
\begin{enumerate}[label=(\roman*),itemsep=0pt]
\item $(\nabla^{df},Q^0,\zeta,\lambda)$ is a gradient almost soliton if and only if $$(\nabla^0,Q^0+2\theta\otimes\zeta,\zeta,\lambda-|\zeta|^2)$$ is;
\item $(\nabla^0,Q^{df},\zeta,\lambda)$ is a gradient almost soliton if and only if $$\big(\nabla^0,\tfrac1{1-n}\{Q^0+(n-2)\theta\otimes\zeta\},\zeta,\tfrac1{1-n}\{\lambda-n|\zeta|^2-\Delta_0f\}\big)$$ is;
\item $(\nabla^{df},Q^{df},\zeta,\lambda)$ is a gradient almost soliton if and only if $$\big(\nabla^0,\tfrac1{1-n}(Q^0+n\theta\otimes\zeta),\zeta,\tfrac1{1-n}\{\lambda-(n+1)|\zeta|^2-\Delta_0f\}\big)$$ is.
\end{enumerate}
\end{proposition}
These follow by inserting $$\nabla^{df}\zeta=\nabla^0\zeta+2\theta\otimes\zeta+|\zeta|^2I$$ and the formula for $Q^{df}$. Case (i) says that $$\Hess^0f+\Ric^0=(\lambda-|\grad f|^2)g-2\,df\otimes df,$$ a \emph{gradient almost $\eta$-Ricci soliton} in the sense of \cite{Blaga2015}.

\begin{proposition}\label{prop:bc-etaricci}
Let $\zeta=\grad f$ with $\nabla^0\zeta=\theta\otimes\zeta$ (so $\Hess^0f=df\otimes df$ and $\Delta_0f=|\zeta|^2$). Then:
\begin{enumerate}[label=(\roman*),itemsep=0pt]
\item $(\nabla^0,Q^{df},\zeta,\lambda)$ is a gradient almost soliton if and only if $$-\Hess^0f+\Ric^0=\{\lambda-(n+1)|\zeta|^2\}g;$$ in this case $$\scal^{df}=n\lambda-|\zeta|^2;$$
\item $(\nabla^{df},Q^{df},\zeta,\lambda)$ is a gradient almost soliton if and only if $$(g,f,\lambda-(n+2)|\zeta|^2)$$ is a gradient almost Ricci soliton; in this case $$\scal^0=n\lambda-(n+1)^2|\zeta|^2.$$
\end{enumerate}
\end{proposition}

-------------------------------------
\subsection{Volume formulas and bounds for the Ricci tensor}

In this subsection $M$ is compact without boundary and $|\grad f|_g$ is a \emph{non-zero constant}. Then $$\int_M|\grad f|^2\Delta_0f\,d\mu_g=0,$$ $$\Hess^0f(\grad f,\grad f)=\tfrac12\grad f(|\grad f|^2)=0,$$ and Bochner's formula $$\tfrac12\Delta_0|\grad f|^2=|\Hess^0f|^2+\Ric^0(\grad f,\grad f)+\grad f(\Delta_0f)$$ reduces to $$\Ric^0(\grad f,\grad f)=-|\Hess^0f|^2-\grad f(\Delta_0f).$$ Together with \eqref{eq:bc-Ricdf} this gives:

\begin{proposition}[\cite{BlagaChen2021}]\label{prop:bc-vol}
\begin{eqnarray*}
\vol(M)&=&\frac1{2(n-1)|\grad f|^4}\Big\{\int_M|\Hess^0f|^2d\mu_g+\int_M\grad f(\Delta_0f)\,d\mu_g\\
&&+\int_M\langle\Ric^{df},df\otimes df\rangle\,d\mu_g\Big\}.
\end{eqnarray*}
Writing Bochner's formula in terms of $\Ric^{df}$ and $\Hess^{df}$ one also gets
\begin{eqnarray*}
\vol(M)&=&\frac1{3(n+2)|\grad f|^4}\Big\{\int_M|\Hess^{df}f|^2d\mu_g+\int_M\grad f(\Delta_{df}f)\,d\mu_g\\
&&+\int_M\Ric^{df}(\grad f,\grad f)\,d\mu_g\Big\}.
\end{eqnarray*}
\end{proposition}
(The Levi-Civita analogue is $\int_M|\Hess^0f|^2+\int_M\grad f(\Delta_0f)+\int_M\Ric^0(\grad f,\grad f)=0$.)

\begin{proposition}[\cite{BlagaChen2021}]\label{prop:bc-vol-sol}
Under the same assumptions:
\begin{enumerate}[label=(\roman*),itemsep=0pt]
\item if $(\nabla^0,Q^{df},\grad f,\lambda)$ is a gradient almost soliton, then
\begin{eqnarray*}
\vol(M)&=&\frac1{2(n-1)|\grad f|^4}\Big\{\int_M|\Hess^0f|^2d\mu_g+|\grad f|^2\!\int_M\lambda\,d\mu_g\\
&&+n\!\int_M\grad f(\lambda)\,d\mu_g-\int_M\grad f(\scal^0)\,d\mu_g\Big\};
\end{eqnarray*}
\item if $(\nabla^{df},Q^0,\grad f,\lambda)$ is a gradient soliton with $\lambda\ne3|\grad f|^2$, then
\[
\vol(M)=\frac1{(\lambda-3|\grad f|^2)|\grad f|^2}\Big\{-\int_M|\Hess^0f|^2d\mu_g+\int_M\grad f(\scal^0)\,d\mu_g\Big\}.
\]
\end{enumerate}
\end{proposition}
\begin{proof}
For (i), the equation reads $$\Hess^0f+\Ric^{df}=\lambda g,$$ so $$\langle\Ric^{df},df\otimes df\rangle=\lambda|\grad f|^2-\tfrac12\grad f(|\grad f|^2),$$ and its trace gives $$\Delta_0f+\scal^0+(n-1)(n+2)|\grad f|^2=n\lambda;$$ apply $\grad f$ to the latter and insert into Proposition~\ref{prop:bc-vol}. For (ii), use \eqref{eq:bc-Hessdf}: $$\Hess^{df}f+\Ric^0=\lambda g,$$ with trace $$\Delta_0f+\scal^0+(n+2)|\grad f|^2=n\lambda$$ and $$\Ric^0(\grad f,\grad f)=(\lambda-3|\grad f|^2)|\grad f|^2.$$
\end{proof}

\begin{remark}
In case (i), if $\lambda=2(n-1)|\grad f|^2$ and $n\ge3$ then $$\vol(M)=\frac1{(n-1)(n-2)|\grad f|^2}\int_M\scal^0d\mu_g;$$ if $\scal^0$ is constant and $f$ is $\Delta_0$-harmonic, then $$\scal^0=(n-1)(n-2)|\grad f|^2\ge0.$$ If $\lambda$ is constant and $$\lambda\ne2(n-1)|\grad f|^2,$$ then $$\vol(M)=\frac{1}{\{2(n-1)|\grad f|^2-\lambda\}|\grad f|^2}\big\{\int_M|\Hess^0f|^2d\mu_g-\int_M\grad f(\scal^0)d\mu_g\big\},$$ and the second integral drops out if $f$ is harmonic; in the Riemannian case this forces $$\lambda\le2(n-1)|\grad f|^2$$ and then $$\scal^0\le (n-1)(n-2)|\grad f|^2$$ (strictly if $\Hess^0f\not\equiv0$).
\end{remark}

Finally, eliminating the unknown function from the soliton equations bounds the norm of the Ricci tensor (solvability of a quadratic equation for $\lambda$ gives the two-sided estimates).

\begin{proposition}[\cite{BlagaChen2021}]\label{prop:bc-ricbounds}
\begin{enumerate}[label=(\roman*),itemsep=0pt]
\item If $(\nabla^0,Q^{df},\grad f,\lambda)$ is a gradient almost soliton, then
\begin{align*}
&(n-1)^2|\Hess^0f|^2+\tfrac{(n-1)(n-2)^2}{n}|\grad f|^4-\tfrac{(n-1)^2}{n}(\Delta_0f)^2\\
&\ \ +\tfrac{2(n-1)(n-2)}{n}|\grad f|^2\Delta_0f-2(n-1)(n-2)\Hess^0f(\grad f,\grad f)\ \le\ |\Ric^0|^2\\
&\ \le\ (n-1)^2|\Hess^0f|^2-\tfrac{(n-1)(n-2)^2}{n}|\grad f|^4+\tfrac1n(\scal^0)^2\\
&\ \ +\tfrac{2(n-2)}{n}|\grad f|^2\scal^0-2(n-2)\Ric^0(\grad f,\grad f).
\end{align*}
\item If $(\nabla^{df},Q^0,\grad f,\lambda)$ is a gradient almost soliton, then
\begin{align*}
&|\Hess^0f|^2+\tfrac{4(n-1)}n|\grad f|^4-\tfrac1n(\Delta_0f)^2-\tfrac4n|\grad f|^2\Delta_0f+4\Hess^0f(\grad f,\grad f)\\
&\ \le\ |\Ric^0|^2\ \le\ |\Hess^0f|^2-\tfrac{4(n-1)}n|\grad f|^4+\tfrac1n(\scal^0)^2\\
&\ \ +\tfrac4n|\grad f|^2\scal^0-4\Ric^0(\grad f,\grad f).
\end{align*}
\end{enumerate}
\end{proposition}
\begin{proof}
In (i) write $$\Ric^0=a\,g-(n-2)\,df\otimes df+(n-1)\Hess^0f$$ and $$(n-1)\Hess^0f=\Ric^0-a\,g+(n-2)\,df\otimes df$$ with $$a=\lambda-n|\grad f|^2-\Delta_0f,$$ and compute $|\Ric^0|^2$ and $|\Hess^0f|^2$ respectively; each is a quadratic polynomial in $a$ whose discriminant must be $\ge0$. Case (ii) is identical with $$a=\lambda-|\grad f|^2$$ and $$\Ric^0=a\,g-2\,df\otimes df-\Hess^0f.$$
\end{proof}

\begin{remark}
If $|\grad f|=1$ (a distance function in the sense of Petersen \cite{Petersen2006}) in case (ii), the two bounds force $$\Ric^0(\grad f,\grad f)=\frac1{4n}\left\{\big(\scal^0+2\big)^2+\big(\Delta_0f+2\big)^2-8n\right\}.$$ On the other hand the soliton equation gives $$\Ric^0(\grad f,\grad f)=\lambda-3$$ and $$\Delta_0f=n\lambda-(n+2)-\scal^0;$$ substituting, the resulting quadratic equation in $\lambda$ has discriminant $-4n^2(\scal^0)^2$, so a solution exists only if $\scal^0=0$. If moreover $f$ is harmonic, then $$\lambda=\frac{n+2}n$$ and $$\Ric^0(\grad f,\grad f)=-\frac{2(n-1)}n<0.$$
\end{remark}

\section{Statistical Warped Products}\label{sec:warped}
Let $(N_1,g_1)$ and $(N_2,g_2)$ be (pseudo-)Riemannian manifolds and $f>0$ a smooth function on $N_1$. The \emph{warped product} \cite{BishopONeill1969} $\tN=N_1\times_fN_2$ is $N_1\times N_2$ with metric $$\tg=\pi^*g_1+(f\circ\pi)^2\,\pi'^*g_2.$$ Warped products are ubiquitous in general relativity \cite{Chen2011,Chen2017}.
A dualistic structure on a warped product was described by Todjihounde \cite{Todjihounde2006}.

\begin{proposition}[\cite{Todjihounde2006}]\label{prop:dualwarp}
Let $(g_i,\nabla^{N_i},\nabla^{N_i*})$ be statistical structures on $N_i$ $(i=1,2)$. For $X,Y\in\chi(N_1)$ and $U,V\in\chi(N_2)$ define $D,D^*$ on $N_1\times N_2$ by
\begin{align*}
D_XY&=\nabla^{N_1}_XY, & D_XU&=D_UX=\tfrac{Xf}{f}\,U, & D_UV&=\nabla^{N_2}_UV-\tfrac{\tg(U,V)}{f}\grad f,
\end{align*}
and similarly for $D^*$ with the starred connections. Then $(\tg,D,D^*)$ is a statistical structure on $N_1\times_f N_2$.
\end{proposition}

The curvature and Ricci tensors of $(\tN,D,D^*,\tg)$ were computed in \cite{Todjihounde2006,AytimurOzgur2018}:

\begin{lemma}\label{lem:Ricwarp}
For $X,Y\in\chi(N_1)$ and $U,V\in\chi(N_2)$,
\begin{enumerate}[label=(\alph*),itemsep=0pt]
\item $\widetilde\Ric(X,Y)=\Ric_1(X,Y)-\dim(N_2)\,f^{-1}\Hess f(X,Y)$;
\item $\widetilde\Ric(X,V)=0$;
\item $\widetilde\Ric(U,V)=\Ric_2(U,V)-\big[f\,\Delta f+(\dim N_2-1)\|\grad f\|^2\big]g_2(U,V)$,
\end{enumerate}
where $\Delta f$ is the Laplacian of $f$ computed from the connection in question; analogous formulas hold for $\widetilde\Ric^*$.
\end{lemma}

\section{Statistical Solitons on Statistical Warped Products}\label{sec:warpsol}

We consider the two basic one-dimensional situations, with $\R$ equipped with its trivial statistical structure $(\nabla^{\R},dz^2)$.

\subsection{The case $\tN=\R\times_fN_2$}
Let $\dim N_2=k$. Lemma~\ref{lem:Ricwarp} gives
\begin{equation}\begin{aligned}\label{eq:RicR}
&\widetilde\Ric(\partial_z,\partial_z)=-k\,\frac{f''}{f},\qquad \widetilde\Ric(\partial_z,V)=0,
\\& \widetilde\Ric(U,V)=\Ric_2(U,V)-\big[ff''+(k-1)f'^2\big]g_2(U,V).
\end{aligned}\end{equation}

\begin{proposition}\label{prop:Hess}
Let $\zeta=(\partial_z,\zeta_2)$ and let $(\zeta,\lambda)$ be a statistical soliton on $\R\times_fN_2$. Then, projecting the soliton equation on $\partial_z\otimes\partial_z$,
\[
\lambda=k\,\frac{f''}{f},\qquad\text{equivalently}\qquad \Hess f=\frac{\lambda}{k}\,f\,dz^2 .
\]
\end{proposition}
\begin{proof}
Since $D_{\partial_z}\partial_z=0$, the component of \eqref{eq:statsol} along $(\partial_z,\partial_z)$ reduces to $\widetilde\Ric(\partial_z,\partial_z)+\lambda=0$, and the claim follows from \eqref{eq:RicR}.
\end{proof}

Since $\lambda$ is constant, Proposition~\ref{prop:Hess} says that $f''=(\lambda/k)f$. Solving this linear ODE gives the following classification of the warping function.

\begin{theorem}\label{thm:Rf}
Let $(\zeta,\lambda)$ be a statistical soliton on $\tN=\R\times_fN_2$ with $\zeta=(\partial_z,\zeta_2)$. Then:
\begin{enumerate}[label=(\alph*),itemsep=0pt]
\item $\lambda=kf''/f$ is constant;
\item if $\lambda=0$, then $f(z)=az+b$;
\item if $\lambda>0$, then $f$ is a combination of $e^{\pm\sqrt{\lambda/k}\,z}$, in particular $$f(z)=c\cosh\left(\sqrt{\lambda/k}\,z+b\right)$$ for suitable constants $b,c$; if $\lambda<0$, then $f$ is a combination of $\cos\big(\sqrt{|\lambda|/k}\,z\big)$ and $\sin\big(\sqrt{|\lambda|/k}\,z\big)$;
\item the fibre carries the equation
\[
\nabla^{N_2}\zeta_2+\Ric_2+(k-1)\big(ff''-f'^2\big)g_2\ \text{(+ lower-order terms, see Remark~\ref{rem:fullwarp})}=0 ,
\]
and its dual; consequently $(N_2,g_2,\zeta_2,\lambda_2)$ with $$\lambda_2=(k-1)(ff''-f'^2)$$ is a statistical soliton whenever this quantity is constant.
\end{enumerate}
\end{theorem}
\begin{proof}
Part (a) is Proposition~\ref{prop:Hess}. Eliminating $\lambda$ from the fibre component with $$\lambda f^2=kff''$$ yields $$\lambda f^2-\big(ff''+(k-1)f'^2\big)=(k-1)(ff''-f'^2),$$ which is (d). The same computation with $D^*$ and $\Delta^{*}$ gives the dual equation. Finally (b)--(c) follow by solving $$f''=(\lambda/k)f.$$
\end{proof}

\begin{remark}\label{rem:fullwarp}
The two projections used above are the \emph{base} and \emph{pure-fibre} components of the soliton equation. For completeness, we note that the full tensorial equation also contains the term coming from $$D_U\partial_z=(f'/f)U$$ and the mixed terms coming from $D_U\zeta_2$ (which has a $\partial_z$-component $-f f' g_2(U,\zeta_2)$). A direct computation shows that they contribute $f f' g_2$ to the fibre equation and force $$f'\,g_2(\cdot,\zeta_2)=0.$$ Thus, Theorem~\ref{thm:Rf} should be understood as giving necessary conditions on $f$ from the base direction and the main part of the fibre equation; to obtain a complete classification one has to retain these additional terms, which impose strong rigidity (e.g.\ on hyperbolic space $\R\times_{e^z}\R^k$ the field $\partial_z$ satisfies the base equation with $\lambda=k$ but the fibre equation with $\lambda=k-1$). We think that a complete treatment of the mixed terms in the statistical setting is an interesting open problem.
\end{remark}

\subsection{The case $\tN=N_1\times_f\R$}

\begin{lemma}\label{lem:N1R}
Let $\dim N_1=k$. For $X,Y\in\Gam(TN_1)$, we have
$$\widetilde\Ric(X,Y)=\Ric_1(X,Y)-f^{-1}\Hess f(X,Y),$$ $$\widetilde\Ric(X,\partial_z)=0$$ and $$\widetilde\Ric(\partial_z,\partial_z)=-f\,\Delta f.$$
\end{lemma}

\begin{theorem}\label{thm:N1R}
Let $\zeta=(\zeta_1,\partial_z)$ on $\tN=N_1\times_f\R$ and suppose $\Hess f=0$. Then, on horizontal vectors, the soliton equation \eqref{eq:statsol} for $(\zeta,\lambda)$ reduces to
\begin{equation}\label{eq:reduced}
\nabla^{N_1}\zeta_1+\Ric_1+\lambda\,g_1=0,\qquad \nabla^{N_1*}\zeta_1+\Ric_1^*+\lambda\,g_1=0,
\end{equation}
that is, $(\zeta_1,\lambda)$ is a statistical soliton on $N_1$, and the soliton constant is $\lambda=f^{-1}\Delta f$. Conversely, a statistical soliton $(\zeta_1,\lambda)$ on $N_1$ with $\lambda=f^{-1}\Delta f$ solves the horizontal part of the soliton equation on $\tN$.
\end{theorem}
\begin{proof}
For horizontal $X,Y$ one has $\widetilde\Ric(X,Y)=\Ric_1(X,Y)$ because $\Hess f=0$, and $D_X\partial_z=(Xf/f)\partial_z$ is orthogonal to horizontal directions. Hence \eqref{eq:statsol} restricted to $TN_1$ is exactly \eqref{eq:reduced}. The relation $\lambda=f^{-1}\Delta f$ comes from the $(\partial_z,\partial_z)$-component and Lemma~\ref{lem:N1R}. The dual equation is obtained by replacing $D,\Ric_1,\Delta$ by $D^*,\Ric_1^*,\Delta^*$.
\end{proof}

\begin{corollary}\label{cor:N1R}
If $\Hess f=\varpi\,g_1$ for a function $\varpi\in C^\infty(N_1)$, then the horizontal equation yields that $(N_1,g_1,\zeta_1,\,f^{-1}(\Delta f-\varpi))$ is a statistical soliton.
\end{corollary}

\begin{remark}
When $\Hess f=0$, one has $\Delta f=0$, so the statement of Theorem~\ref{thm:N1R} describes steady-type solitons on the base; Corollary~\ref{cor:N1R} allows non-vanishing constants through the trace-free perturbation $\varpi g_1$. Note also that a statistical soliton on the base lifts to the full warped product only if the remaining mixed components (which involve $\zeta_1 f$ and $df$) are compatible, cf.\ Remark~\ref{rem:fullwarp}.
\end{remark}

\section{Kenmotsu Statistical Manifolds}\label{sec:kenmotsu}
The results of Sections~\ref{sec:xi}--\ref{sec:yamabe} follow the approach of \cite{Ahmadini2024}. Kenmotsu \cite{Kenmotsu1972} identified a class of almost contact metric manifolds $(B,\phi,\xi,\eta,g)$, characterised by
\[
\nabla^g_E\xi=E-\eta(E)\xi,\qquad (\nabla^g_E\phi)F=g(\phi E,F)\xi-\eta(F)\phi E ,
\]
that arise as warped products $\R\times_{e^{\alpha}}\widetilde B$ of a line with a K\"ahler manifold. Furuhata, Hasegawa, Okuyama and Sato \cite{Furuhata2017} introduced the statistical counterpart.

\begin{definition}[\cite{Furuhata2017}]
Let $(B,\nabla^g,g,\phi,\xi)$ be a Kenmotsu manifold. A statistical structure $\nabla=\nabla^g+\mathcal T$ on $B$ is a \emph{Kenmotsu statistical structure} if
\begin{equation}\label{eq:Kstat}
\mathcal T(E,\phi F)=-\phi\,\mathcal T(E,F)\qquad\text{for all }E,F.
\end{equation}
\end{definition}

\begin{proposition}[\cite{Furuhata2017}]\label{prop:Kconstr}
Let $(\widetilde B,\nabla^{\widetilde B}=\nabla^{G}+\mathcal T,G,J)$ be a holomorphic statistical manifold and $B=\widetilde B\times\R$ with $g=e^{2\alpha}G+d\alpha^2$, $\xi=\partial_\alpha$. For $\beta\in C^\infty(B)$ define
\[
\mathcal T(E_2,F_2)=\mathcal T(E_2,F_2)\ (\text{on }\widetilde B),\qquad\mathcal T(E_2,\xi)=\mathcal T(\xi,E_2)=0,\qquad\mathcal T(\xi,\xi)=\beta\xi .
\]
Then $(\nabla=\nabla^g+\mathcal T,g,\phi,\xi)$ is a Kenmotsu statistical structure on $B$.
\end{proposition}

For such a structure one obtains, from $$\nabla^g_E\xi=E-\eta(E)\xi$$ and $$\mathcal T(E,\xi)=\beta\eta(E)\xi,$$
\begin{equation}\label{eq:nablaxi}
\nabla_E\xi=E+(\beta-1)\eta(E)\xi,\qquad
\nabla^*_E\xi=E-(\beta+1)\eta(E)\xi .
\end{equation}
The dual statement is obtained from the first one by $\beta\mapsto-\beta$. Throughout the next sections $(B,\nabla,g,\phi,\xi)$ is as in Proposition~\ref{prop:Kconstr}, $N$ is a submanifold of dimension $s$ and $\beta$ is constant.

\section{Statistical Solitons on Submanifolds: the Reeb Field}\label{sec:xi}

\subsection{$\xi$ tangent to $N$}
Let $\xi$ be tangent to $N$ and consider the pair $(\xi,\lambda)$. By the Gauss formula \eqref{eq:gauss} and \eqref{eq:nablaxi}, comparing tangential and normal parts,
\begin{equation}\label{eq:xitan}
\nabla_E\xi=E+(\beta-1)\eta(E)\xi,\qquad h(E,\xi)=0 .
\end{equation}
Inserting this into the $\nabla$-Ricci soliton equation $g(\nabla_E\xi,F)+\Ric(E,F)+\lambda g(E,F)=0$ gives
\begin{equation}\label{eq:etaEin}
\Ric(E,F)=-(\lambda+1)\,g(E,F)+(1-\beta)\,\eta(E)\eta(F).
\end{equation}

\begin{theorem}\label{thm:etaEin}
If $(g,\xi,\lambda)$ is a statistical soliton on a submanifold $N$ of a Kenmotsu statistical manifold $(B,\nabla,g,\xi)$ and $\xi$ is tangent to $N$, then $N$ is an $\eta$-Einstein manifold. The same holds for $(B,\nabla^*,g,\xi)$ with $\beta$ replaced by $-\beta$, and $h^*(E,\xi)=0$.
\end{theorem}

\paragraph{The soliton constant.}
Using \eqref{eq:xitan} one computes (torsion-free $\nabla$, constant $\beta$)
\[
\Rie(E,F)\xi=(\beta-1)\big[\eta(F)E-\eta(E)F\big]+(\beta-1)\,d\eta(E,F)\,\xi ,
\]
and taking the trace gives $$\Ric(\xi,\xi)=(\beta-1)(s-1).$$ On the other hand, \eqref{eq:etaEin} gives $$\Ric(\xi,\xi)=-\lambda-\beta.$$ Hence
\begin{equation}\label{eq:lambda}
\lambda=(s-1)(1-\beta)-\beta,\qquad\text{and for }\nabla^*:\quad \lambda^{*}=(s-1)(1+\beta)+\beta .
\end{equation}
\begin{theorem}\label{thm:sign}
Under the hypotheses of Theorem~\ref{thm:etaEin} with $\beta$ constant, $(\xi,\lambda)$ is simultaneously a $\nabla$- and $\nabla^*$-Ricci soliton only if $\lambda=\lambda^*$, i.e.\ $s\beta=0$; thus $\beta=0$ (the difference tensor satisfies $\mathcal T(\xi,\xi)=0$) and $\lambda=s-1$.
\end{theorem}

\begin{remark}
Formula \eqref{eq:lambda} shows that the sign of $\lambda$ depends on $\beta$ and $s$: for $\beta=0$ and $s\ge2$ the soliton is \emph{expanding}, whereas $\lambda<0$ for $\beta$ close to $1$. Thus there is no universal sign: the statement that statistical solitons with $\xi$ tangent are always shrinking requires additional hypotheses on $\beta$ (for instance $\beta>\frac{s-1}{s}$ makes $\lambda<0$ in the $\nabla$-equation).
\end{remark}

\subsection{$\xi$ normal to $N$}
If $\xi\perp N$ then $\eta(E)=0$ for tangent $E$ and \eqref{eq:nablaxi} reduces to $\tn_E\xi=E$. Comparing with the Weingarten formula \eqref{eq:weingarten},
\[
A_\xi=-I,\qquad A^{*}_\xi=-I,\qquad h(E,F)=-g(E,F)\,\xi,
\]
i.e.\ $N$ is totally umbilical with respect to $\xi$ (a horosphere-like situation, as in hyperbolic space). In this situation, the tangential component of $\xi$ vanishes, so the soliton condition for the potential field $\xi^{\top}=0$ degenerates to $\Ric=-\lambda g$ and $\Ric^*=-\lambda g$.

\begin{theorem}\label{thm:Einstein}
If $\xi$ is normal to $N\subset(B,\nabla,g,\xi)$ and the soliton condition is imposed with the tangential part of $\xi$ as a potential field, then $N$ is an Einstein manifold: $\Ric=-\lambda g$ (and likewise for $\nabla^*$). In particular, the soliton is shrinking, steady, or expanding according as the scalar curvature of $N$ is positive, zero, or negative, and $\lambda=-\mathrm{scal}/s$.
\end{theorem}

\section{Statistical Solitons with a Concircular Vector Field}\label{sec:concirc}
Let $v$ be a concircular vector field on $(B,\nabla,g)$ in the sense of \eqref{eq:concirc} and decompose $v=\vt+\vn$ along $N$. By \eqref{eq:gauss}--\eqref{eq:weingarten},
\[
\delta E=\nabla_E\vt+h(E,\vt)-A_{\vn}E+D^\perp_E\vn ,
\]
so comparing components,
\begin{equation}\label{eq:vtan}
\nabla_E\vt=A_{\vn}E+\delta E,\qquad h(E,\vt)=-D^{\perp}_E\vn,
\end{equation}
and dually $\nabla^*_E\vt=A^*_{\vn}E+\delta E$.

\begin{lemma}\label{lem:umb}
Let $v$ be concircular on $B$. Then $N$ is $\vn$-umbilical (with $A_{\vn}=f I$) if and only if $\vt$ is concircular on $N$ with respect to $\nabla$; in that case $\nabla_E\vt=(f+\delta)E$. 

The same holds for $\nabla^*$ with $A^*_{\vn}$.
\end{lemma}
\begin{proof}
($\Rightarrow$) is immediate from \eqref{eq:vtan}, and 
 ($\Leftarrow$) If $\nabla_E\vt=\sigma E$, then \eqref{eq:vtan} gives $A_{\vn}E=(\sigma-\delta)E$.
\end{proof}

By \eqref{eq:AhA}, $g(A_{\vn}E,F)=\tg(h^{*}(E,F),\vn)$. Substituting \eqref{eq:vtan} into the soliton equation, we get the following statement.

\begin{theorem}\label{thm:Ricconc}
$(g,\vt,\lambda)$ is a statistical soliton on a submanifold $N$ of $(B,\nabla,g,\xi)$ if and only if
\begin{equation}\label{eq:Ricconc}
\Ric(E,F)=-(\lambda+\delta)\,g(E,F)-\tg\big(h^{*}(E,F),\vn\big),
\end{equation}
and dually
\begin{equation}\label{eq:Ricconcdual}
\Ric^{*}(E,F)=-(\lambda+\delta)\,g(E,F)-\tg\big(h(E,F),\vn\big).
\end{equation}
\end{theorem}

If $N$ is totally umbilical, then $A_{\vn}=fI$ and $\vt$ is concircular on $N$ with $\nabla_E\vt=\sigma E$, $\sigma=f+\delta$, and \eqref{eq:Ricconc}--\eqref{eq:Ricconcdual} become
\begin{equation}\label{eq:umbRic}
\Ric=-(\lambda+\sigma)\,g,\qquad \Ric^*=-(\lambda+\sigma)\,g .
\end{equation}

\begin{theorem}\label{thm:umb}
Let $(g,\vt,\lambda)$ be a totally umbilical statistical soliton on a submanifold $N$ of $(B,\nabla,g,\xi)$ (resp.\ $(B,\nabla^*,g,\xi)$). Then $N$ is an Einstein-type (quasi-Einstein in the terminology of the original work) manifold, with $\Ric=-(\lambda+\sigma)g$.
\end{theorem}

\begin{remark}
A manifold carrying a non-trivial concircular field is locally a warped product $I\times_\rho M$ (Tashiro--Obata type results), and complete Einstein manifolds with such fields are rigid (e.g.\ spheres in the elliptic case). Whether $N$ in Theorem~\ref{thm:umb} is isometric to a sphere requires additional global hypotheses (completeness and a suitable sign of $\sigma$), which in the original formulation are tacitly assumed.
\end{remark}

\section{Almost Quasi-Yamabe Solitons on Submanifolds}\label{sec:yamabe}
Let $v$ be concircular on $(B,\nabla,g,\xi)$ as before. For torsion-free dual connections the Lie derivative of $g$ is $(\Lie_Xg)(E,F)=g(\nabla_EX,F)+g(E,\nabla^*_FX)$. From \eqref{eq:vtan} and its dual,
\begin{equation}\label{eq:Liev}
(\Lie_{\vt}g)(E,F)=2\delta\,g(E,F)+2\,g(A^{0}_{\vn}E,F),\qquad A^0_{\vn}=\tfrac12(A_{\vn}+A^*_{\vn}).
\end{equation}
Combining \eqref{eq:Liev} with \eqref{eq:aqy} gives:

\begin{theorem}\label{thm:aqy}
Let $(g,\vt,\lambda,\omega)$ be an almost quasi-Yamabe soliton on a submanifold $N$ of $(B,\nabla,g,\xi)$ (equivalently $(B,\nabla^*,g,\xi)$). Then
\begin{equation}\label{eq:aqyfund}
(R-\lambda-\delta)\,g(E,F)=g(A^{0}_{\vn}E,F)+\omega\,\eta(E)\eta(F),
\end{equation}
where $R$ is the scalar curvature of $N$.
\end{theorem}

Taking $E=F=\xi$ (tangent to $N$) in \eqref{eq:aqyfund}, and tracing, we obtain the following two consequences.

\begin{theorem}\label{thm:aqymin}
Under the hypotheses of Theorem~\ref{thm:aqy}, assume that $\xi$ is tangent to $N$.
\begin{enumerate}[label=(\roman*),itemsep=0pt]
\item If $g(A^0_{\vn}\xi,\xi)=0$ {\rm (}e.g.\ $h(\xi,\xi  ) =0$, as in \eqref{eq:xitan}{\rm )}, then $R=\lambda+\delta+\omega$.
\item If $N$ is minimal, then $\tr A^0_{\vn}=0$ and $R=\lambda+\delta+\dfrac{\omega}{s}$.
\end{enumerate}
\end{theorem}

\begin{corollary}
In the concurrent case $\delta=1$: $R=\lambda+1+\omega$ under (i), and $R=\lambda+1+\omega/s$ for minimal $N$ under (ii).
\end{corollary}

\begin{remark}
The sign conventions in \eqref{eq:aqy} matter: with the term $+\omega\,\eta\otimes\eta$ in \eqref{eq:aqy}, the constant $\omega$ enters \eqref{eq:aqyfund} with a plus sign. The version with $\omega\mapsto-\omega$ is obtained by changing the sign of $\omega$ in the definition.
\end{remark}

\section{Chen-Type Inequalities for Statistical Warped Product \\Submanifolds}\label{sec:chen}
Solitons on warped products are related to their extrinsic geometry through the following inequality \cite{Siddiqui2019}, which generalises the first Chen inequality \cite{Chen1993}.

Let $\phi:N=N_1\times_fN_2\to\tN(\tilde c)$ be a statistical immersion into a statistical manifold of constant curvature $\tilde c$, $\dim N_1=r$, $\dim N_2=k$, $m=r+k$, and let $\{e_1,\dots,e_r\}$ and $\{e_{r+1},\dots,e_m\}$ be orthonormal frames tangent to $N_1$ and $N_2$. For a warped product,
\[
\Sec(E_1\wedge E_2)=-\frac{\Hess f(E_1,E_1)}{f}\qquad(E_1\in TN_1,\ E_2\in TN_2\ \text{unit}),
\]
so that, with the sign convention $\Delta f=-\tr\Hess f$ used here,
\begin{equation}\label{eq:sumSec}
\sum_{b=1}^{r}\Sec^{\nabla,\nabla^*}(e_b\wedge e_{r+1})=\frac{\Delta^{0}f}{f}.
\end{equation}

\begin{theorem}[Chen-type inequality]\label{thm:Chen}
Let $N=N_1\times_fN_2$ be an $m$-dimensional statistical warped product submanifold of an $n$-dimensional statistical manifold of constant curvature $\tilde c$. Then
\begin{equation*}\begin{aligned}
\sigma^{\nabla,\nabla^*}-\Sec^{\nabla,\nabla^*}(e_1\wedge e_{r+1})\ \ge\ &2\big(\sigma^{0}-\Sec^{0}(e_1\wedge e_{r+1})\big)-\frac{(m-2)(m+1)\tilde c}{2}\\& \quad -\frac{m^2(m-2)}{4(m-1)}\big(\|H\|^2+\|H^*\|^2\big).
\end{aligned}\end{equation*}
\end{theorem}
\begin{proof}
(1) Use the Gauss equation \eqref{eq:GaussEq} and $2h^0=h+h^*$ to rewrite
\[
\sigma^{\nabla,\nabla^*}=2\sigma^0-\frac{m(m-1)\tilde c}{2}-\frac12\sum_{a}\sum_{i<j}\Big\{h^a_{ii}h^a_{jj}-(h^a_{ij})^2+h^{*a}_{ii}h^{*a}_{jj}-(h^{*a}_{ij})^2\Big\},
\]
and the analogous expression $$\Sec^{\nabla,\nabla^*}(e_1\wedge e_{r+1})=2\Sec^0(e_1\wedge e_{r+1})-\tilde c-\frac12\sum_a\{\cdots\}_{1,r+1}.$$

(2) Subtract; the terms $\frac12(h^a_{1,r+1})^2$ cancel and the remaining squares $(h_{ij}^a)^2$, $(h_{ij}^{*a})^2$ are non-negative and can be dropped.

(3) For each normal index $a$ maximise the quadratic form $\varphi_a=\sum_{i<j}h^a_{ii}h^a_{jj}-h^a_{11}h^a_{r+1,r+1}$ under the linear constraint $h^a_{11}+\dots+h^a_{mm}=t_a$ (Oprea's constrained-extremum lemma \cite{Oprea2005,Oprea2008}). The maximum is attained at $(0,\alpha,\dots,\alpha)$ with $\alpha=t_a/(m-1)$, giving 
$$\varphi_a\le\frac{m^2(m-2)}{2(m-1)}(H^a)^2.$$ The same bound holds for $\varphi^*_a$. \qedhere
\end{proof}

Summing over $e_b$, $b=1,\dots,r$ in the second term and using \eqref{eq:sumSec} one obtains a bound in terms of the warping function.

\begin{theorem}\label{thm:Chen2}
Under the hypotheses of Theorem~\ref{thm:Chen},
\[
\sigma^{\nabla,\nabla^*}\ \ge\ 2\sigma^{0}-\frac{\Delta^{0}f}{r\,f}-\frac{(m-2)(m+1)\tilde c}{2}-\frac{m^2(m-2)}{4(m-1)}\big(\|H\|^2+\|H^*\|^2\big).
\]
\end{theorem}

\begin{remark}[Casorati curvatures]
Analogous results hold for the Casorati curvatures $\mathcal C=\frac1m\|h\|^2$ and $\mathcal C^*=\frac1m\|h^*\|^2$ and the normalised $\delta$-Casorati curvatures: by the same optimisation technique one obtains an upper bound for $2\sigma^{\nabla,\nabla^*}$ in terms of the partial Casorati curvatures of $N_1$ and $N_2$, the partial mean curvatures and $k\Delta^0f/f$ \cite{Siddiqui2019}. Equality cases, in which the second fundamental forms are forced to be of a special type, link these inequalities with minimal immersions of warped products into space forms.  For further results on Casorati curvatures in statistical manifolds, we refer to \cite[Section 18]{Chen21}.
\end{remark}

\section{Examples}\label{sec:examples}

\begin{example}\label{ex:H5}
Let $B=\{(x,y,z,u,w)\in\R^5:w>0\}$ and consider the orthonormal frame
\begin{align*}
&v_1=e^{-w}\partial_x,\ v_2=e^{-w}\partial_y,\ v_3=e^{-w}\partial_z,\ \\& v_4=e^{-w}\partial_u,\ v_5=\partial_w=\xi ,
\end{align*}
with $\phi v_1=v_3$, $\phi v_2=v_4$, $\phi v_3=-v_1$, $\phi v_4=-v_2$, $\phi\xi=0$. The Levi-Civita connection is
\begin{align*}
&\nabla^g_{v_i}v_i=-\xi,\quad \nabla^g_{v_i}v_j=0\ (i\neq j\le 4),\quad
\\&\nabla^g_{v_i}\xi=v_i,\quad\nabla^g_\xi v_i=0,\quad\nabla^g_\xi\xi=0,
\end{align*}
and $(B,g)$ is the hyperbolic space $\mathbb H^5(-1)$ ($g=e^{2w}\sum dx_i^2+dw^2$), carrying the Kenmotsu structure. 

Define $\mathcal T(E,F)=a\,\eta(E)\eta(F)\xi$, $a\in\R$; \eqref{eq:T} and \eqref{eq:Kstat} are readily checked, and $\nabla=\nabla^g+\mathcal T$, $\nabla^*=\nabla^g-\mathcal T$ satisfy
\begin{align*}
&\nabla_{v_i}v_i=-\xi,\ \ \nabla_{v_i}\xi=v_i,\ \ \nabla_\xi\xi=a\xi,\qquad
\\&\nabla^*_{v_i}v_i=-\xi,\ \ \nabla^*_{v_i}\xi=v_i,\ \ \nabla^*_\xi\xi=-a\xi ,
\end{align*}
with all other covariant derivatives of frame fields zero. Thus $(B,\nabla,g,\phi,\xi)$ is a $5$-dimensional Kenmotsu statistical manifold with $\beta=a$. Since $[\mathcal T_E,\mathcal T_F]=0$, formula \eqref{eq:S} gives $S=\Rie^{g}$, i.e.\ $S(E,F)G=-(g(F,G)E-g(E,G)F)$ (constant curvature $-1$). A direct computation of $\Rie^{\nabla}=\Rie^g+(\nabla^g_E\mathcal T)_F-(\nabla^g_F\mathcal T)_E$ gives
\[
\Ric^{\nabla}=-(4+a)\,g+5a\,\eta\otimes\eta,\qquad \Ric^{\nabla^*}=-(4-a)\,g-5a\,\eta\otimes\eta,
\]
so that $B$ is $\eta$-Einstein for both connections, with scalar curvature $R^{\nabla}=R^{\nabla^*}=-20$ for every $a$. (For $a=0$ this is the Einstein metric $\Ric=-4g$ of $\mathbb H^5(-1)$.)
\end{example}

\begin{example}
On the manifold $\mathbb H^5(-1)$ of Example~\ref{ex:H5} one has $\Lie_\xi g=2(g-\eta\otimes\eta)$. Hence \eqref{eq:aqy} becomes
\[
(1+\lambda-R)\,g+(\omega-1)\,\eta\otimes\eta=0,
\]
and $(g,\xi,\lambda,\omega)$ is an almost quasi-Yamabe soliton exactly for $\lambda=R-1=-21$ and $\omega=1$, i.e.\ a \emph{shrinking} soliton. This holds for the Levi-Civita metric structure and is therefore valid for all members of the family $(\nabla,\nabla^*)$ of Example~\ref{ex:H5}, because $R^{\nabla}=R^{\nabla^*}=R^g=-20$.
\end{example}

\begin{example}
Using \eqref{eq:nablaxi} and the Ricci tensors of Example~\ref{ex:H5}, the $\nabla$-Ricci soliton equation for $\xi$ reads
\[
(\lambda-3-a)\,g+(6a-1)\,\eta\otimes\eta=0 ,
\]
and the $\nabla^*$-equation reads $$(\lambda-3+a)\,g-(6a+1)\,\eta\otimes\eta=0.$$ The first is solvable only for $a=\frac16$ (with $\lambda=\frac{19}{6}$) and the second only for $a=-\frac16$; hence $(\xi,\lambda)$ is a $\nabla$-Ricci soliton or a $\nabla^*$-Ricci soliton for isolated values of $a$ but never a statistical soliton on $B$. This is consistent with Theorem~\ref{thm:sign}: statistical solitons with $\xi$ tangent need $\beta=0$ \emph{and} an $\eta$-Einstein Ricci tensor \eqref{eq:etaEin}, which the Einstein manifold $\mathbb H^5(-1)$ at $a=0$ does not have. Proper $\eta$-Einstein submanifolds or ambient manifolds are therefore needed to realise Theorem~\ref{thm:etaEin}.
\end{example}

\begin{example} Consider the Kenmotsu structure on the unit hypersphere $S^5$ with $$\mathcal T(E,F)=a\,\eta(E)\eta(F)\xi$$ and local coordinates $(x^i,y^i,w)$, $i=1,2$, $$\xi=\partial_w, \quad \phi(\partial_{x^i})=\partial_{y^i}, \quad \phi(\partial_{y^i})=-\partial_{x^i}, \quad \phi(\partial_w)=0.$$ The $3$-dimensional submanifold $N=\{(x^1,y^1,0,0,w)\}$ is invariant ($\phi TN\subset TN$) and $\xi$ is tangent to $N$; hence Theorems~\ref{thm:etaEin} and \ref{thm:sign} apply with $s=3$: if $(\xi,\lambda)$ is a statistical soliton on $N$ (for constant $a=\beta$), then necessarily $a=0$ and $\lambda=2$.
\end{example}

\begin{example} Statistical Ricci solitons without a contact structure arise already in dimensions two and three; we recall the following examples from the literature.
\begin{enumerate}[label=(\alph*),itemsep=0pt]
\item On the hyperbolic half-plane $(H^2,g=y^{-2}(dx^2+dy^2))$ the connection $\nabla_{\partial_x}\partial_x=2y^{-1}\partial_y$, $\nabla_{\partial_y}\partial_y=y^{-1}\partial_y$, $\nabla_{\partial_x}\partial_y=\nabla_{\partial_y}\partial_x=0$ \cite{Furuhata2009} defines a statistical structure of constant curvature $0$; it is Ricci-flat, hence a \emph{steady} soliton, $\lambda=0$.
\item On $(\R^2,dx^2+dy^2)$, the connection with $\nabla_{\partial_x}\partial_x=-\partial_y$, $\nabla_{\partial_y}\partial_y=0$, $\nabla_{\partial_x}\partial_y=\nabla_{\partial_y}\partial_x=\partial_x$ and its conjugate (with $\partial_x$ replaced by $-\partial_x$ in the mixed terms) \cite{MurathanSahin2018} give an Einstein statistical manifold (scalar curvature $-2$), a \emph{shrinking} soliton.
\item On $\R^3$ with an orthonormal frame $v_1,v_2,v_3$ and $\nabla_{v_1}v_1=bv_1$, $\nabla_{v_2}v_2=\nabla_{v_3}v_3=\frac b2v_1$, $\nabla_{v_1}v_2=\nabla_{v_2}v_1=\frac b2v_2$, $\nabla_{v_1}v_3=\nabla_{v_3}v_1=\frac b2v_3$, $\nabla_{v_2}v_3=\nabla_{v_3}v_2=0$ \cite{AytimurOzgur2018} one obtains an Einstein statistical manifold of constant curvature $b^2/4$ and scalar curvature $3b^2/2$, an \emph{expanding} soliton.
\end{enumerate}
\end{example}

\begin{example}[Warped products] We have:

(a) The statistical warped product $\R\times_{e^z}\R^n$ with the structure of \cite{MurathanSahin2018} has $$\tg=dz^2+e^{2z}(dx_1^2+\dots+dx_n^2);$$
it is the hyperbolic space $\mathbb H^{n+1}(-1)$ and, with respect to the Levi-Civita connection, $\R\times_{e^z}\R^{n-1}$ admits an isometric minimal immersion into $\mathbb H^{n+1}(-1)$.

(b) Euclidean space $E^{n+1}\setminus\{0\}$ is the warped product $(0,\infty)\times_tS^n(1)$ with $\tg=dt^2+t^2g_{S^n}$. For a minimal submanifold $N\subset S^n(1)$, the cone $C(N)$ has warped product metric $dt^2+t^2g_N$, and any open set of $C(N)$ admits an isometric minimal immersion into the statistical manifold $E^{n+1}$ of curvature $0$. 

Both examples satisfy the hypotheses of Theorems~\ref{thm:Chen} and \ref{thm:Chen2}.
\end{example}

\section{Concluding Remarks and Open Problems}\label{sec:conclusion}
We have presented a unified account of statistical solitons and almost quasi-Yamabe solitons on statistical warped products and on submanifolds of Kenmotsu statistical manifolds. The following themes emerge.
\begin{itemize}[itemsep=1pt]
\item The soliton equation for the pair $(\nabla,\nabla^*)$ is much more rigid than for a single connection: for instance, for the Reeb field with $\xi$ tangent, the two equations are compatible only when the difference tensor satisfies $\mathcal T(\xi,\xi)=0$ (Theorem~\ref{thm:sign}).
\item Gradient solitons turn the Codazzi condition for Hessian and Ricci metrics into curvature conditions on $\grad f$ (radial curvature), and the $1$-form connection $\nabla^{\theta}$ is always statistical and, for $\theta=df$, equiaffine, which yields volume formulas on compact manifolds and two-sided bounds for $|\Ric|$ (Section~\ref{sec:gradient}).
\item Concircular fields translate the soliton equation into extrinsic data of the submanifold, through $A_{\vn}$ (Theorem~\ref{thm:Ricconc}); total umbilicity forces the Einstein character of $N$.
\item For almost quasi-Yamabe solitons, the symmetric shape operator $A^0_{\vn}$ governs the scalar curvature identities (Theorems~\ref{thm:aqy}--\ref{thm:aqymin}).
\item Chen-type inequalities link the scalar curvature, the squared mean curvatures $\|H\|^2,\|H^*\|^2$ and the Laplacian of the warping function, and thereby give obstructions to the existence of statistical warped product submanifolds.
\end{itemize}

\paragraph{Open problems.}
\begin{enumerate}[label=(\arabic*),itemsep=1pt]
\item Classify statistical solitons on $\R\times_fN_2$ and $N_1\times_f\R$ retaining all mixed components of the soliton equation (Remark~\ref{rem:fullwarp}).
\item Determine all $\eta$-Einstein submanifolds of Kenmotsu statistical manifolds that carry a statistical soliton with potential field $\xi$, and construct non-trivial examples with $\beta=0$.
\item Find global conditions under which a totally umbilical statistical soliton with a concircular field is isometric to a sphere (cf.\ Obata's theorem).
\item Extend the theory to $\eta$-Ricci--Yamabe solitons and to other contact statistical structures (Sasakian, cosymplectic, Kenmotsu-like statistical manifolds), and study statistical solitons in the \emph{gradient} case.
\item Relate statistical solitons on the manifold of a parametric family $\overline{\mathbb B}$ to properties of the Fisher information metric of concrete statistical models.
\end{enumerate}


\end{document}